\documentclass[a4paper]{amsart}
\usepackage{verbatim}
\usepackage[dvips]{color}
\usepackage{amssymb}
\usepackage[active]{srcltx}

\usepackage{latexsym}
\usepackage{amsmath}
\usepackage{float}
\usepackage{amsrefs}
\usepackage{esint}
\usepackage{enumerate}

\newtheorem{theorem}{Theorem}[section]

\newtheorem{lemma}{Lemma}[section]

\theoremstyle{definition}
\newtheorem{definition}{Definition}

\floatstyle{ruled} \restylefloat{figure} \restylefloat{table}

\newcommand{\tail}{\operatorname{Tail}}
\newcommand{\ep}{\varepsilon}

\numberwithin{equation}{section}

\def\R{\mathbb R}

\def\A{\mathcal A}

\def\L{\mathcal L}

\def\0{\boldsymbol 0}

\newcommand{\supp}{\operatorname{supp}}

\newtoks\by
\newtoks\paper
\newtoks\book
\newtoks\jour
\newtoks\yr
\newtoks\pages
\newtoks\vol
\newtoks\publ

\def\name[#1, #2]{#1 #2}
\def\ota{{\hbox{\bf ???}}}
\def\cLear{\by=\ota\paper=\ota\book=\ota\jour=\ota\yr=\ota
\pages=\ota\vol=\ota\publ=\ota}
\def\endpaper{\the\by, \textit{\the\paper},
{\the\jour} \textbf{\the\vol} (\the\yr), \the\pages.\cLear}
\def\endbook{\the\by, \textit{\the\book},
\the\publ, \the\yr.\cLear}
\def\endpap{\the\by, \textit{\the\paper}, \the\jour.\cLear}
\def\endproc{\the\by, \textit{\the\paper}, \the\book, \the\publ,
\the\yr, \the\pages.\cLear}

\begin{document}

\title[]{Local boundedness of solutions to non-local  parabolic equations modeled on the fractional $p-$Laplacian}


\author[Str{\"o}mqvist]{Martin Str{\"o}mqvist
}
\address{Martin  Str{\"o}mqvist\\ Department of Mathematics, Uppsala University\\
S-751 06 Uppsala, Sweden}
\email{martin. str{\"o}mqvist@math.uu.se}


\maketitle
\begin{abstract}
	\noindent  We state and prove estimates for the local boundedness of subsolutions of non-local, possibly degenerate, parabolic integro-differential equations of the form
\begin{equation*}
	 \partial_tu(x,t)+\mbox{P.V.}\int\limits_{\mathbb R^n}K(x,y,t) |u(x,t)-u(y,t) |^{p-2}(u(x,t)-u(y,t))\, dy,\end{equation*}
$(x,t)\in\mathbb R^n\times\mathbb R$, where $\mbox{P.V.} $ means in the principle value sense,  $p\in (1,\infty)$ and the kernel obeys $K(x,y,t)\approx  |x-y |^{n+ps}$ for some $s\in (0,1)$, uniformly in $(x,y,t)\in\mathbb R^n\times \mathbb R^n\times\mathbb R$.
\medskip
\end{abstract}

\noindent

\medskip


 \setcounter{equation}{0} \setcounter{theorem}{0}

\section{Introduction and statement of main results}
\noindent
In this work we study local regularity properties of solutions to the equation 
\begin{equation}\label{eq1}
\frac{\partial u(x,t)}{\partial t} + Lu(x,t) = 0\quad\text{in }\Omega\times(t_1,t_2),
\end{equation}
for a bounded domain $\Omega$.
In \eqref{eq1}, $L$ is a nonlinear, nonlocal operator of $p$-Laplace type. Specifically, we assume that $L$ is formally given by  
\begin{equation}\label{eqPV}
Lu(x,t) = \text{P.V.}\int_{R^n}|u(x,t)-u(y,t)|^{p-2}(u(x,t)-u(y,t))K(x,y,t)dy,
\end{equation}
where P.V. means principal value and the kernel $K$ satisfies, for some $\Lambda\ge 1$ and $s\in(0,1)$,
\begin{equation}\label{ellipticity}
\frac{\Lambda^{-1}}{|x-y|^{n+sp}} \le K(x,y,t)\le \frac{\Lambda}{|x-y|^{n+sp}}.
\end{equation}
Throughout the paper we will assume that $p\ge 2$, which corresponds to equations that are possibly degenerate. 

Elliptic nonlocal equations of this type $(Lu=0)$ has received great attention in recent years. Ishii and Nakamura \cite{IshiiNak} were the first authors to study this equation, with $K(x,y,t) = (1-s)|x-y|^{n+sp}$ and in a localized setting. They proved existence and uniqueness of viscosity solutions and showed that in this case $L$ converges to the $p$-Laplace operator as $s\to1$. 
In \cite{DKPpmin} Di Castro, Kuusi and Palatucci studied the elliptic counterpart of \eqref{eq1} and proved local boundedness and H\"older continuity of solutions. In \cite{DKPharnack} the same authors proved a very interesting nonlocal version of the Harnack inequality for solutions $u$. It involves the so-called tail of the negative part of $u$ and does not require solutions to be globally positive. 
Through the use of fractional DeGiorgi classes, M. Cozzi \cite{Cozzi} proved the results of \cite{DKPpmin} and \cite{DKPharnack} for solutions to a more general class of equations, involving a term $f(u)$, or solutions to associated minimum problems. 

When it comes to parabolic problems, an analogous theory of local boudedness, H\"older continuity and Harnack's inequality does not exist for $p\neq 2$. 
In the linear case $p=2$, Felsinger and Kassmann \cite{FelKass} prove a weak Harnack inequality and H\"older continuity for weak solutions to \eqref{eq1} that are globally positive. They work with a class of kernels satisfying slightly weaker growth conditions than \eqref{ellipticity}. Due to the assumption of global positivity, the nonlocal term involving the negative part of the solution (the tail term), that normally occur in such estimates, is not present. In \cite{SchwKass}, Schwab and Kassmann prove results similar to those in \cite{FelKass}, but with $a(t,x,y)d\mu(x,y)$ in place of $K(t,x,y)dxdy$, merely assuming that $\mu$ is a measure, not necessarily absolutely continuous w.r.t.\ Lebesgue measure, that satisfies certain growth conditions. It should also be mentioned that the conditions on imposed on the kernels/measures in \cite{FelKass} and \cite{SchwKass} are in general not sufficient to prove a Harnack inequality. This is due to a result by Bogdan and Sztonyk \cite{BogSzton} that prove sharp conditions on the kernel for a Harnack inequaity to hold (in the elliptic setting). 
To the authors best knowledge, there is as of yet no theory of local boundedness for equations of the type \eqref{eq1}, even when $p=2$. 
However, the situation is different if the equation \eqref{eq1} holds globally in space. 
Caffarelli, Chan and Vasseur \cite{CCV} study parabolic nonlocal, nonlinear equations of quadratic growth in all space. They prove that solutions are bounded and H\"older continuous as soon as the initial data is in $L^2$. 

The purpose of this paper is to to develop a basis for further study of the regularity theory of weak solutions to equations of the type \eqref{eq1}. To this end we prove Cacciopollo type inequalities and establish local boundedness of weak subsolutions. In future projects we will study Harnack/H\"older estimates for \eqref{eq1}. 

H\"older estimates and Harnack inequalities for local equations of $p$-Laplace type is considerably more involved in the parabolic setting, compared to the elliptic setting, or to the parabolic setting for $p=2$. This is essentially due to the inherent inhomogeneity of these equations, which leads to intrinsic Harnack/H\"older estimates that are valid only for times depending on the local size of the solution. Harnack's inequality for local equations was proved independently by Kuusi \cite{Kuusi} and DiBenedetto, Gianazza and Vespri \cite{DiBGV}. The results in \cite{Kuusi} were modified and extended to a wider class of operators in \cite{ACCN} by Avelin, Capogna, Citti and Nystr\"om.  For H\"older estimates we refer to \cite{DiB}.


Our main result is that local weak solutions to \eqref{eq1} are bounded. The estimates will depend on a nonlocal quantity called the parabolic tail of the solution. If $v\in L^p(t_0-T_0,t_0;W^{s,p}(\R^n))$, the (parabolic) tail of $v$ is defined by 
\[
\tail(v;x_0,r,t_1-T_1,t_1) = \left(\frac{r^{sp}}{T_1}\int_{t_1-T_1}^{t_1}\int_{\R^n\setminus B_r(x_0)}\frac{|v(x,t)|^{p-1}}{|x-x_0|^{n+sp}}dxdt\right)^{\frac{1}{p-1}},
\]
whenever $t_1\le t_0$ and $t_0-T_0\le t_1-T_1$. 
If $Q = B_r(x_0)\times(t_1-T_1,t_1)$, we set 
\[\tail(v;Q)=\tail(v;x_0,r,t_1-T_1,t_1).\]
At times we will use a supremum (in time) version of the tail, given by 
\[
\tail_\infty(v;x_0,r,t_1-T_1,t_1) = \left(r^{sp}\sup_{t_1-T_1<t<t_1}\int_{\R^n\setminus B_r(x_0)}\frac{|v(x,t)|^{p-1}}{|x-x_0|^{n+sp}}dx\right)^{\frac{1}{p-1}}. 
\]
For parabolic rescaling of cubes $Q$, we will use the notation 
\[\lambda Q = B_{\lambda r}(x_0)\times(t_1-\lambda^{sp}T_1,t_1).\]
In all our estimates, $C\ge 1$ will denote a generic constant that depends only on $n$  and $p$ unless otherwise stated. The numerical value of $C$ may change during the course of an estimate. 
We can now state our main theorem. 
\begin{theorem}\label{main_thm}
Let $Q = B_R(x_0)\times(t_0-T_0,t_0)$ and suppose that $u$ is a nonnegative subsolution in $2Q$. 
Then, if $p>2$ 
\begin{align}\label{supu1}
\sup_{\sigma Q}u &\le \frac{C}{(1-\sigma)^\alpha}\left(\left(\frac{R^{sp}}{T_0}\right)^{\frac{1}{p-2}} + \frac{T_0}{R^{sp}}\tail^{p-1}_\infty(u_+;x_0,\sigma r,t_0-T_0,t_0)\right)\\
&+ \frac{C}{(1-\sigma)^\alpha}\frac{T_0}{R^{sp}}\left(\sup_{t_0-T_0<t<t_0}\fint_{B_R}u(x,t)dx \right)^{p-1},\notag
\end{align}
for any $\sigma\in (0,1)$. 
\end{theorem}
We remark that that if $\tail_\infty(u_+;x_0,\sigma r,t_0-T_0,t_0)\le C\left(\frac{R^{sp}}{T_0}\right)^{\frac{1}{p-2}}$, then 
\begin{align*}
&\frac{T_0}{R^{sp}}\tail^{p-1}_\infty(u_+;x_0,\sigma r,t_0-T_0,t_0)\\
&\le \tail_\infty(u_+;x_0,\sigma r,t_0-T_0,t_0) \le \left(\frac{R^{sp}}{T_0}\right)^{\frac{1}{p-2}}.
\end{align*}
Then \eqref{supu1} becomes 
\begin{align}
&\sup_{\sigma Q}u \le \frac{C}{(1-\sigma)^\alpha}\left(\left(\frac{R^{sp}}{T_0}\right)^{\frac{1}{p-2}} + \left(\frac{T_0}{R^{sp}}\sup_{t_0-T_0<t<t_0}\fint_{B_R}u(x,t)dx \right)^{p-1}\right),\notag
\end{align}
This is precisely the estimate that holds for solutions to local equations. 

\subsection{Parabolic Sobolev spaces}

For a domain $D\subset\R^n$, the fractional Sobolev space $W^{s,p}(D)$ consists of all functions $f\in L^p(D)$ such that 
\[
[f]_{W^{s,p}(D)} = \int_D\int_D\frac{|f(x)-f(y)|^p}{|x-y|^{n+sp}}dxdy<\infty. 
\]
The norm of $f\in W^{s,p}(D)$ is given by 
\[
\|f\|_{W^{s,p}(D)} = \|f\|_{L^p(D)} + \left(\int_D\int_D\frac{|f(x)-f(y)|^p}{|x-y|^{n+sp}}dxdy\right)^{\frac1p}.
\]
We shall also need the space 
\[
W_0^{s,p}(\Omega) = \{f\in W^{s,p}(\R^n): f=0\text{ in }\R^n\setminus\Omega\}, 
\]
endowed with the norm $\|\cdot\|_{W^{s,p}(\R^n)}$. 
We will later use the fact that a truncation of $f$ does not increase its norm in $W^{s,p}$: 
\begin{align}
&[f_+]_{W^{s,p}(\Omega)}\le [f]_{W^{s,p}(\Omega)},\label{trunk1}\\
&
[\min\{f,m\}]_{W^{s,p}(\Omega)}\le [f]_{W^{s,p}(\Omega)},\quad\text{for any }m\in\R.\label{trunk2}
\end{align}
To prove \eqref{trunk1} we need only note that $|a_+-b_+|\le |a-b|$ for any $a,b\in\R$. 
Then \eqref{trunk2} is a consequence of \eqref{trunk1} and the fact that $\min\{f,m\} = -(m-f)_++m$. 

For the fractional Sobolev embedding below we refer to \cite{Hitch}. 
\begin{theorem}[Sobolev embedding]\label{sobolev}
Suppose $p\ge 1$, $sp<n$ and let $p*=\frac{np}{n-sp}$. Then 
for any $f\in W^{s,p}(\R^n)$ and $q\in\left[p,p*\right]$,
\begin{equation}\label{sobRn}
\|f\|_{L^q(\R^n)}^p\le \int_{\R^n}\int_{\R^n}\frac{|f(x)-f(y)|^p}{|x-y|^{n+sp}}dxdy.
\end{equation}
If $\Omega$ is an extension domain for $W^{s,p}$, then
\begin{equation}\label{sobOmega}
\|f\|_{L^q(\Omega)}\le C(\Omega)\|f\|_{W^{s,p}(\Omega)}.
\end{equation}
If $sp = n$, then \eqref{sobRn} and \eqref{sobOmega} hold for any $q\in[p,\infty)$. 
\end{theorem}

\begin{lemma}\label{parabSobolev}
Suppose $p\ge 1$, $sp<n$ and let $\kappa^* = \frac{n}{n-sp}$ and suppose that 
\[f\in L^p(t_1,t_2;W^{s,p}_0(B_r)).\] 
Then for any $\kappa\in[1,\kappa^*]$,
\begin{align}
&\int_{t_1}^{t_2}\fint_{B_r}|f|^{\kappa p}dxdt\\
&\le Cr^{sp-n}\int_{t_1}^{t_2}[f(\cdot, t)]_{W^{s,p}(B_r)}^pdt
\times\left(\sup_{t_1<t<t_2}\fint_{B_r}|f|^{\frac{p\kappa^*(\kappa-1)}{\kappa^*-1}}dx\right)^{\frac{\kappa^*-1}{\kappa^*}}.\notag
\end{align}
\end{lemma}
\begin{proof}
The proof is identical to the proof of the analogous statement for the space $L^p(t_1,t_2;W^{1,p}_0(B_r))$, c.f. \cite{DiB}. We need only note that for a function $g\in W^{s,p}_0(B_r)$, its extension by zero to $\R^n$ belongs to $W^{s,p}(\R^n)$ and we are at liberty to apply \eqref{sobRn}.   
\end{proof}

\subsection{Weak Solutions}
We are now in a position to define weak solutions, and will show that for any bounded domain $\Omega\subset \R^n$ and $T>0$, the problem  
\begin{equation}\label{eqdata}
\left\{\begin{aligned}
&\frac{\partial u(x,t)}{\partial t} + Lu(x,t) = 0,\quad\text{in }\Omega_{T} = \Omega\times(0,T),\\
&u(x,t) = g(x,t),\quad\text{in }(\R^n\setminus\Omega)\times(0,T),\\
&u(x,0) = u_0(x),\quad\text{in }\R^n,
\end{aligned}\right.
\end{equation}
has a unique solution in a suitable sense, whenever $g$ and $u_0$ belong to appropriate function spaces. 
Motivated by \eqref{eq1} and \eqref{eqPV}, we define a weak solution as follows. For the sake of brevity we will use the notation 
\begin{align*}
&\A u(x,y,t) = K(x,y,t)|u(x,t)-u(y,t)|^{p-2}(u(x,t)-u(y,t)),\\
&\delta u(x,y,t) = u(x,t)-u(y,t)\\
&d\mu = d\mu(x,y,t) = K(x,y,t)dxdydt.
\end{align*}

\begin{definition}
Suppose 
\begin{align*}
&g\in L^p(0,T;W^{s,p}(\R^n)),\\
&\partial_tg\in L^{p'}(0,T;(W^{s,p}(\R^n))^*),\\ 
&u_0\in L^2(\Omega). 
\end{align*}
We say that 
$u\in L^p((0,T);W^{s,p}(\R))$ is a \emph{weak solution} to \eqref{eqdata} if 
\begin{align*}&\partial_tu\in L^{p'}(0,T;(W^{s,p}(\R^n))^*),\\
&u-g\in L^p(0,T;W_0^{s,p}(\Omega
))
\end{align*}
and 
\begin{align}\label{weak1}
&\int_{t_1}^{t_2}\int_{\R^n}\int_{\R^n}\A u(x,y,t)(\eta(x,t)-\eta(y,t))dxdydt -\int_{t_1}^{t_2}\int_{R^n}u\partial_t\eta dxdt\\
&=\int_{\Omega}u_0\eta(0,x) dx,\notag
\end{align}\label{weak2}
for any $\eta\in L^p((0,T);W_0^{s,p}(\Omega))$ such that 
\[
\partial_t\eta\in L^{p'}(0,T;(W^{s,p}(\R^n))^*) \quad\text{and}\quad\eta(x,t_2) = 0.
\]
Let $w= u-g$. Then $u$ solves \eqref{weak1} if and only if $w\in L^p(0,T;W^{s,p}_0(\Omega))$ solves 
\begin{align}\label{weak10}
&\int_{t_1}^{t_2}\int_{\R^n}\int_{\R^n}\A (w+g)(x,y,t)(\eta(x,t)-\eta(y,t))dxdydt\\
&-\int_{t_1}^{t_2}\int_{R^n}w\partial_t\eta dxdt = 
\int_{t_1}^{t_2}\int_{R^n}g\partial_t\eta dxdt + \int_{\Omega}u_0\eta(0,x) dx,\notag
\end{align}\label{weak2}
for any $\eta\in L^p((0,T);W_0^{s,p}(\Omega))$ such that 
\[
\partial_t\eta\in L^{p'}(0,T;(W^{s,p}(\R^n))^*) \quad\text{and}\quad\eta(x,t_2) = 0.
\]

\end{definition}

\subsubsection{Wellposedness}
The existence and uniqueness of a solution to \eqref{weak10} is a consequence of the general theory for degenerate parabolic equations in Banach spaces, see \cite{Show}. We will only briefly explain the properties of the equation that need to be verified. Let $\tilde\A(\cdot) = \A(\cdot+g)$. 
Suppose $u(\cdot,t)$ and $v(\cdot,t)$ belong to $W_0^{s,p}(\Omega)$. Then by H\"older's inequality and \eqref{ellipticity}, 
\begin{align}\label{opL}
&\int_{\R^n}\int_{\R^n}\tilde\A u(x,y,t)(v(x,t)-v(y,t))dxdydt \\
&\le \Lambda[u(\cdot,t)+g(\cdot,t)]_{W^{s,p}}^{p-1}[v(\cdot,t)]_{W^{s,p}}\notag\\
&\le \Lambda 2^{p-1}\|u\|_{W^{s,p}}^{p-1}\|v\|_{W^{s,p}} + \Lambda 2^{p-1}\|g\|_{W^{s,p}}^{p-1}\|v\|_{W^{s,p}}.\notag 
\end{align}
Thus $\tilde \A$ defines an operator $\mathcal{L}_t:W^{s,p}(\R^n)\mapsto (W^{s,p}(\R^n))^*$, with $\langle \mathcal{L}_t u,v\rangle$ given by \eqref{opL} and 
\begin{equation}\label{Lbound}
\|\mathcal{L}_t u\| \le 2^{p-1}\Lambda \|u(\cdot,t)\|_{W^{s,p}}^{p-1} + 2^{p-1}\Lambda \|g(\cdot,t)\|_{W^{s,p}}^{p-1}. 
\end{equation}
Additionally, $\L_t$ is a monotone operator, i.e.\ 
\[
\langle \mathcal{L}_t u-\mathcal{L}_tv,u-v\rangle\ge 0,\quad\text{for all }u,v\in W^{s,p}(\R^n).  
\]
Indeed, 
\begin{align*}
\langle \mathcal{L}_t u-\mathcal{L}_tv,u-v\rangle& = \langle \mathcal{L}_t u-\mathcal{L}_tv,u+g-(v+g)\rangle\\
&=\int_{\R^n}\int_{\R^n}|\delta(u+g)|^pd\mu + \int_{\R^n}\int_{\R^n}|\delta(v+g)|^pd\mu \\
&- \int_{\R^n}\int_{\R^n}|\delta(u+g)|^{p-2}\delta(u+g)\delta(v+g)d\mu\\
&- \int_{\R^n}\int_{\R^n}|\delta(v+g)|^{p-2}\delta(v+g)\delta(u+g)d\mu\\
& \ge \int_{\R^n}\int_{\R^n}|\delta(u+g)|^pd\mu + \int_{\R^n}\int_{\R^n}|\delta(v+g)|^pd\mu \\
&-\frac{p-1}{p}\int_{\R^n}\int_{\R^n}|\delta(u+g)|^pd\mu - \frac{1}{p}\int_{\R^n}\int_{\R^n}|\delta(v+g)|^pd\mu\\
&-\frac{p-1}{p}\int_{\R^n}\int_{\R^n}|\delta(v+g)|^pd\mu - \frac{1}{p}\int_{\R^n}\int_{\R^n}|\delta(u+g)|^pd\mu=0, 
\end{align*}
where we used Young's inequality. 
The existence of a unique weak solution now follows from Proposition 4.1. in \cite{Show} if, in addition to \eqref{Lbound} and the monotonicity, we prove that  
\begin{align}
&[u]_{s,p}\ge \alpha\|u\|_{W^{s,p}(\R^n)},\label{seminormisnorm}\\
& \langle \mathcal{L}_t u,u\rangle \ge \alpha[u]_{W^{s,p}}^p - C[g]_{W^{s,p}}^p,\text{ for some }\alpha>0.\label{coerc}  
\end{align}
The Sobolev inequality guarantees that \eqref{seminormisnorm} holds. 
Let us prove \eqref{coerc}. By Young's inequality with $\ep$ and \eqref{ellipticity}, 
\begin{align*}
&\langle \mathcal{L}_t u,u\rangle = \langle \mathcal{L}_t u,u+g-g\rangle\\
& = \int_{\R^n}\int_{\R^n}|\delta (u+g)|^pd\mu - \int_{\R^n}\int_{\R^n}|\delta (u+g)|^{p-2}\delta (u+g)\delta gd\mu\\
&\ge\Lambda^{-1}[u+g]_{W^{s,p}}^p - \ep\Lambda [u+g]_{W^{s,p}}^p
-C(\ep)\Lambda [g]_{W^{s,p}}^p
\end{align*}
Choosing $\ep = \frac{1}{2\Lambda^2}$, we obtain 
\begin{align*}
\langle \mathcal{L}_t u,u\rangle & \ge\frac{1}{2\Lambda}[u+g]_{W^{s,p}}^p
 -C\Lambda[g]_{W^{s,p}}^p \\
 & \ge \frac{1}{2^{p+1}\Lambda}[u]_{W^{s,p}}^p - \left(\frac{1}{2\Lambda} + C\Lambda\right)[g]_{W^{s,p}}^p, 
\end{align*}
from which \eqref{coerc} follows. 
The initial data $u_0$ is assumed in the sense that 
\[
\lim_{t\to0}\int_{\Omega}|u(x,t)-u_0|^2dx=0. 
\]
The reason for choosing $u_0\in L^2(\Omega)$ is that $u_0$ needs to be an element of a Hilbert space $H$ such that $W^{s,p}_0(\Omega)$ is dense and continuously embedded into $H$. This is indeed true because of the sobolev embedding theorem and the fact that $C_c^\infty(\Omega)$ is dense in $W_0^{s,p}(\Omega)$.


\section{Estimates for subsolutions}
\noindent
%
\begin{definition}
We say that $u$ is a solution to $\partial_t u + Lu = 0$ in $\Omega\times(t_1,t_2)$ if 
\begin{align}\label{fullsol}
&\int_{t_1}^{t_2}\int_{\R^n}\int_{\R^n}\A u(x,y,t)(\eta(x,t)-\eta(y,t))dxdydt -\int_{t_1}^{t_2}\int_{R^n}u\partial_t\eta dxdt = 0,
\end{align}
for all $\eta\in L^p(t_1,t_2;W^{s,p}_0(\Omega))$ such that $\partial_t\eta\in L^{p'}(t_1,t_2;(W^{s,p}(\Omega))^*)$ and $\eta(x,t_1)=\eta(x,t_2)$, for all $x\in\Omega$. 
\end{definition}
\begin{definition}
We say that $u$ is a subsolution to $\partial_t u + Lu = 0$ in $\Omega\times(t_1,t_2)$ if 
\begin{align}\label{subsol}
&\int_{t_1}^{t_2}\int_{\R^n}\int_{\R^n}\A u(x,y,t)(\eta(x,t)-\eta(y,t))dxdydt -\int_{t_1}^{t_2}\int_{R^n}u\partial_t\eta dxdt \le 0,
\end{align}
for all $\eta$ as in the definition of a solution that are also non negative. 
\end{definition}
We first prove that if $u$ is a subsolution, then its positive part, $u_+ = \max\{u,0\}$, is again a subsolution. 
\begin{lemma}\label{trunksubsol}
If $u$ is a subsolution to \eqref{weak1}, then $u_+$ is also a subsolution. 

\end{lemma}

\begin{proof}
Let $\phi_j(\tau)$ be a smooth, convex approximation of $\tau_+$ such that $\phi_j(\tau)=0$ if $\tau\le -1/j$,  $\phi_j(\tau),\;\phi_j'(\tau)>0$ if $\tau>-1/j$ and $|\phi_j'|\le C$, $|\phi_j''|\le C(j)$. Let $v$ be a non negative, bounded test function. Then $\phi_j'(u)v$ is an admissible test function, as can be easily seen from the following equality
\begin{align*}
&\phi_j'(u(x,t))v(x,t) - \phi_j'(u(y,t))v(y,t) \\
&= \frac12(\phi_j'(u(x,t))-\phi_j'(u(y,t)))(v(x,t)+v(y,t))\\
& + \frac12(v(x,t)-v(y,t))(\phi_j'(u(x,t))+\phi_j'(u(y,t))). 
\end{align*}
Let $\phi_j(x,t) = \phi_j(u(x,t))$ and let $\phi_j'(x,t) = \phi_j'(u(x,t))$. We also set 
\begin{equation}
u_{j,+}(x,t) = \max\{u(x,t),-1/j\} = 
\left\{\begin{array}{l}
u(x,t)\text{ if }\phi_j'(x,t)>0,\\
-1/j\text{ if }\phi_j'(x,t)=0.
\end{array}\right.
\end{equation}
Using $\phi_j'(u)v$ as a test function in \eqref{subsol} we obtain
\begin{align*}
&\int_{t_1}^{t_2}\int_{\Omega}\partial_tu\phi_j'(u)vdxdt + \int_{t_1}^{t_2}\int_{\R^n}\int_{\R^n}\A u(x,y,t)\delta(\phi_j'(u)v)(x,y,t)dxdydt\\
& = I_{1,j}+I_{2,j} \le 0. 
\end{align*}
We may write $I_{1,j}$ as 
\begin{equation}\label{timeuplus}
I_{1,j} = \int_{t_1}^{t_2}\int_{\Omega}v\partial_t\phi_j(u)dxdt\to I_{1} = \int_{t_1}^{t_2}\int_{\Omega}v\partial_t u_+dxdt, \quad\text{as }j\to\infty.
\end{equation}
We next estimate the integrand of $I_{2,j}$ under the assumption that $u(x,t)>u(y,t)$. 
If $\phi_j'(x,t)=0$, then $\A u(x,y,t)\delta(\phi_j'(u)v)(x,y,t) = 0$ since $\phi'$ is monotone non decreasing. 
If $\phi_j'(y,t)>0$, then 
\begin{align*}
&(u(x,t)-u(y,t))^{p-1}(\phi_j'(x,t)v(x,t)-\phi_j'(y,t)v(y,t))\\
& =(u_{j,+}(x,t)-u_{j,+}(y,t))^{p-1}(\phi_j'(x,t)v(x,t)-\phi_j'(y,t)v(y,t))\\
&\ge(u_{j,+}(x,t)-u_{j,+}(y,t))^{p-1}\phi_j'(x,t)(v(x,t)-v(y,t)).
\end{align*}
If $\phi_j'(y,t)=0$ and $\phi_j'(x,t)>0$, then 
\begin{align*}
& (u(x,t)-u(y,t))^{p-1}(\phi_j'(x,t)v(x,t)-\phi_j'(y,t)v(y,t))\\
& = (u(x,t)-u(y,t))^{p-1}\phi_j'(x,t)v(x,t)\\
&\ge (u_{j,+}(x,t)-u_{j,+}(y,t))^{p-1}\phi_j'(x,t)v(x,t)\\
&\ge (u_{j,+}(x,t)-u_{j,+}(y,t))^{p-1}\phi_j'(x,t)(v(x,t)-v(y,t)). 
\end{align*}
We have thus shown that if $u(x,t)>u(y,t)$, 
\begin{align}
&\A u(x,y,t)\delta(\phi_j'(u)v)(x,y,t)\label{phi(x)}\\
&\ge K(x,y,t)(u_{j,+}(x,t)-u_{j,+}(y,t))^{p-1}\phi_j'(x,t)(v(x,t)-v(y,t)).\notag
\end{align}
By interchanging the roles of $x$ and $y$, we obtain, for $u(x,t)<u(y,t)$, the analogous estimate
\begin{align}
&\A u(x,y,t)\delta(\phi_j'(u)v)(x,y,t)\label{phi(y)}\\
&\ge K(x,y,t)(u_{j,+}(y,t)-u_{j,+}(x,t))^{p-1}\phi_j'(y,t)(v(y,t)-v(x,t))\notag\\
& = K(x,y,t)|u_{j,+}(x,t)-u_{j,+}(y,t)|^{p-2}\notag\\
&\quad\times(u_{j,+}(x,t)-u_{j,+}(y,t))\phi_j'(y,t)(v(x,t)-v(y,t)).\notag
\end{align}
Since the expressions in \eqref{phi(x)} and \eqref{phi(y)} are $L^1(\R^n\times\R^n\times(t_1,t_2))$, we obtain 
\begin{align*}
&\liminf_{j\to\infty}I_{2,j}\\
&\ge \int_{t_1}^{t_2}\int_{\R^n}\int_{\R^n}K(x,y,t)(u_{+}(x,t)-u_{+}(y,t))^{p-1}(v(x,t)-v(y,t))dxdydt\\
& = \int_{t_1}^{t_2}\int_{\R^n}\int_{\R^n}\A u_+(x,y,t)\delta v(x,y,t)dxdydt.
\end{align*}
In combination with \eqref{timeuplus}, this gives 
\[
\int_{t_1}^{t_2}\int_{\Omega}v\partial_t u_+dxdt + \int_{t_1}^{t_2}\int_{\R^n}\int_{\R^n}\A u_+(x,y,t)\delta v(x,y,t)dxdydt\le 0,
\]
for all bounded, non negative test functions $v$, and by a standard approximation argument, all non negative test functions $v$.

\end{proof}

\subsection{Caccioppoli estimate} Let $\zeta_h(s)$ be a standard mollifier with support in $(-h,h)$. Given $f : \R^n \times \R \to \R$, we define
\begin{equation}\label{mollify}
	f_h(x,t) = \int_{\R} f(x,s) \zeta_h(t-s) ds.
\end{equation}
\begin{definition}\label{lebesgue-instant}
	Let $\Omega \subset \R^n$ be a domain, $u \in L^p(t_1,t_2;W^{s,p}(\Omega))$, and consider $t_1 < t < t_2$. Then $t$ is called a Lebesgue instant for $u$ if
	\begin{equation*}
		\lim_{h \to 0} \int_{\Omega} |u_h (x,t)-u(x,t)|^2 dx = 0.
	\end{equation*}
\end{definition}
Since $\int_{\Omega} u (x,t) dx$ belongs to $L^p(t_1,t_2)$, it follows from Lebesgue's differentiation theorem that a.e.\ $t\in (t_1,t_2)$ is a Lebesgue instant. 

\begin{lemma} \label{lem_cacc}
	Let $p\in (1,\infty)$, $s\in (0,1)$. Let $\xi \ge 1$ and assume that $K$ satisfies the ellipticity condition \eqref{ellipticity}. Let $x_0\in\mathbb R^n$, $\tau_1<\tau_2$, $B_r:=B_r(x_0)$, and assume that $u $ is a non-negative sub-solution in $B_r\times (\tau_1, \tau_2)$. Let
$t_1,t_2$ be Lebesgue instants for $u$, with $\tau_1<t_1<t_2<\tau_2$. For $d>0$, let $v=u+d$, $w=v^{(p-1+\xi)/p}$. Then
	\begin{align*}
		&\int_{t_1}^{t_2} \int_{B_r} \int_{B_r}  {|w\phi(x,t)-w\phi(y,t)|^{p}}\, d\mu + \frac{1}{\xi+1} \int_{B_r}
v(x,t)^{1+\xi} \phi^p(x,t) dx \bigg |_{t=t_1}^{t_2} \\
&\leq C \int_{t_1}^{t_2} \int_{B_r} \int_{B_r} \max\{w(x,t), w(y,t)\}^p|\phi(x,t)-\phi(y,t)|^p\, d\bar\mu\notag\\
&+C\biggl (\sup_{x\in\supp\psi}\int_{\R^n\setminus B_r}|x-y|^{-(n+ps)}\, dy\biggr)\biggl (\int_{t_1}^{t_2}\int_{B_r}w^p(x,t)\phi^p(x,t)\, dx dt\biggr )\\
&+C\int_{t_1}^{t_2}\left(\sup_{x\in\supp\psi}\int_{\R^n\setminus B_r}\frac{u(y,t)_+^{p-1}}{|x-y|^{n+sp}}dy\int_{B_r}v^\xi\phi^p(x,t)dx\right)dt\\
		&+ \frac{1}{(1+\xi)}\int_{t_1}^{t_2} \int_{B_r}v^{1+\xi} \left (  \frac{\partial \phi^p}{\partial t} \right )_{+} dx dt,
	\end{align*}
	for all $\phi(x,t) = \psi(x) \zeta(t)$ with $\zeta \in C_0^{\infty}(\tau_1,\tau_2)$ and $\psi \in C_0^{\infty}(B_r)$.
	\end{lemma}
 \begin{proof} 

Let 
\[
v = u+d,\quad v_m = \min\{v,m\},\quad m\ge d,
\]
and let $\phi$ be as in the statement of the theorem. Let $q = 1-\xi\le 0$. Then  $\eta = v_m^{1-q}\phi^p$ is an admissible test function. This is clear if $q=0$. If $q<0$, it is enough to note that, according to the mean value theorem, 
\[
|v_m^{1-q}(x,t)-v_m^{1-q}(y,t)| = (1-q)\alpha^{-q}|v_m(x,t)-v_m(y,t)|, 
\]
for some $v_m(y,t)<\alpha< v_m(x,t)$. For $\tau_1<t_1<t_2<\tau_2$, let $\theta_j(t)\in C_c^\infty(\tau_1,\tau_2)$ be a smooth approximation of $\chi_{(t_1,t_2)}$ as $j\to\infty$. We will test the equation \eqref{weak2} with the function
\begin{equation}\label{etah}
\eta_{j,h} = \left((v_m)_h^{1-q}\phi^p\theta_j\right)_h,
\end{equation}
where the subscript $h$ on the right hand side denotes mollification in the sense of \eqref{mollify}. 
Hence we obtain 
\begin{align}\label{Ih}
0&\ge \int_{\tau_1}^{\tau_2}\int_{B_r}\int_{B_r}\A v(x,y,t)(\eta_{j,h}(x,t)-\eta_{j,h}(y,t))dxdydt\\
&+2\int_{\tau_1}^{\tau_2}\int_{\R^n\setminus B_r}\int_{B_r}\A v(x,y,t)\eta_{j,h}(x,t)dxdydt\notag\\
&- \int_{\tau_1}^{\tau_2}\int_{B_r}v\frac{\partial\eta_{j,h}}{\partial t}dxdt = I_1^{j,h}+ I_2^{j,h} + I_3^{j,h}.\notag 
\end{align}
For $I_3^{j,h}$ we have 
\begin{align}\label{I3}
I_3^{j,h} &= -\int_{\tau_1}^{\tau_2}\int_{B_r} (v_m)_h\partial_t((v_m)_h^{1-q}\phi^p\theta_j)dxdt\\
&= \int_{\tau_1}^{\tau_2}\int_{B_r}\partial_t (v_m)_h(v_m)_h^{1-q}\phi^p\theta_jdxdt\notag\\
&\to \int_{t_1}^{t_2}\int_{B_r}\partial_t (v_m)_h(v_m)_h^{1-q}\phi^pdxdt = I_3^h,\notag
\end{align}
as $j\to\infty$. Then integration by parts yields 
\begin{align}
&I_3^h = \int_{t_1}^{t_2}\int_{B_r}\partial_t \frac{(v_m)_h^{2-q}}{2-q}\phi^pdxdt\\
&=\frac{1}{1+\xi}\int_{B_r}(v_m)_h^{1+\xi}(x,t)\phi^p(x,t)dx\bigg |_{t=t_1}^{t_2} -  \frac{1}{1+\xi}\int_{t_1}^{t_2}\int_{B_r}(v_m)_h^{1+\xi}\partial_t\phi^pdxdt\notag\\
&\to \frac{1}{1+\xi}\int_{B_r}v_m^{1+\xi}(x,t)\phi^p(x,t)dx\bigg |_{t=t_1}^{t_2} - \frac{1}{1+\xi}\int_{t_1}^{t_2}\int_{B_r}v_m^{1+\xi}\partial_t\phi^pdxdt,\notag
\end{align}
as $h\to0$.
Since $I_1^{j,h}$ and $I_2^{j,h}$ are finite, our taking $j\to\infty$ in these terms simply replaces $\tau_i$ by $t_i$. By standard properties of mollifiers, we may then pass to the limit $h\to0$ in \eqref{Ih} and obtain 
\begin{align}\label{subsolnoh}
0&\ge \int_{t_1}^{t_2}\int_{B_r}\int_{B_r}\A v(x,y,t)(\eta(x,t)-\eta(y,t))dxdydt\\
&+2\int_{t_1}^{t_2}\int_{\R^n\setminus B_r}\int_{B_r}\A v(x,y,t)\eta(x,t)dxdydt\notag\\
&+ \frac{1}{1+\xi}\int_{B_r}v_m^{1+\xi}(x,t)\phi^p(x,t)dx\bigg |_{t=t_1}^{t_2} - \frac{1}{1+\xi}\int_{t_1}^{t_2}\int_{B_r}v_m^{1+\xi}\partial_t\phi^pdxdt\notag\\ 
&= I_1+ I_2 + I_3.\notag 
\end{align}
We start by estimating the integrand of $I_1$ under the assumption that $v(x,t)>v(y,t)$. For such $x,y,t$ we may apply the truncation result \eqref{trunk2}, or rather its short proof, to $\A v(x,y,t)$, to find 
\begin{align}
\A v(x,y,t)(\eta(x,t)-\eta(y,t)) &\ge \A v_m(x,y,t)(\eta(x,t)-\eta(y,t))\\
&=\A v_m(x,y,t)(v_m\phi^p(x,t)-v_m\phi^p(y,t)).\notag
\end{align}
In order to simplify notation, we will write $v$ rather than $v_m$ in the estimation of $I_1$. 
We will make use of the inequality 
\begin{equation}
\phi^p(y,t) \le \phi^p(x,t) + c_p\ep\phi^p(x,t) + (1+c_p\ep)\ep^{1-p}|\phi(x,t)-\phi(y,t)|^p,
\end{equation} 
valid for any $\ep\in (0,1)$, see Lemma 3.1 in \cite{DKPpmin}. We let $\delta\in (0,1)$ be a parameter to be chosen and set 
\[
\ep = \delta\frac{v(x,t)-v(y,t)}{v(x,t)}. 
\]
Thus we obtain, 
\begin{align}
&\A v(x,y,t)\left(\frac{\phi^p(x,t)}{v^{q-1}(x,t)}-\frac{\phi^p(y,t)}{v^{q-1}(y,t)}\right)\\
&\ge \A v(x,y,t)\left(\frac{\phi^p(x,t)}{v^{q-1}(x,t)} -\frac{\phi^p(x,t)}{v^{q-1}(y,t)}\left(1+c_p\delta\frac{v(x,t)-v(y,t)}{v(x,t)}\right)\right)\notag\\
&-\frac{\A v(x,y,t)}{v^{q-1}(y,t)}\left(\left(1+c_p\delta\frac{v(x,t)-v(y,t)}{v(x,t)}\right)\delta^{1-p}\frac{(v(x,t)-v(y,t))^{1-p}}{v^{1-p}(x,t)}\right)\notag\\
&\times|\phi(x,t)-\phi(y,t)|^p = D + E.\notag 
\end{align}
We first estimate $D$ and note that 
\begin{align}\label{eqDg}
D & = \A v(x,y,t)\frac{\phi^p(x,t)}{v^{q-1}(y,t)}\left(\frac{v^{q-1}(y,t)}{v^{q-1}(y,t)} - 1 -c_p\delta\frac{v(x,t)-v(y,t)}{v(x,t)}\right)\\
& = \A v(x,y,t)\frac{\phi^p(x,t)}{v^{q}(y,t)}(v(x,t)-v(y,t))\notag\\
&\times\left(\frac{v^q(y,t)}{v^{q-1}(x,t)(v(x,t)-v(y,t))} - \frac{v(y,t)}{v(x,t)-v(y,t)} - c_p\delta\frac{v(y,t)}{v(x,t)}\right)\notag\\
& =  K(x,y,t)\frac{\phi^p(x,t)}{v^{q}(y,t)}(v(x,t)-v(y,t))^p\left(\frac{\frac{v^{q-1}(y,t)}{v^{q-1}(x,t)}-1}{\frac{v(x,t)}{v(y,t)}-1} - c_p\delta\frac{v(y,t)}{v(x,t)}\right).\notag 
\end{align}
For $a>1$, let 
\[
g(a) = \frac{a^{1-q}-1}{a-1},
\]
so that for $v(x,t)>v(y,t)$, 
\[
g\left(\frac{v(x,t)}{v(y,t)}\right) = \frac{\frac{v^{q-1}(y,t)}{v^{q-1}(x,t)}-1}{\frac{v(x,t)}{v(y,t)}-1}. 
\]
Since $\xi\ge1$ we have $q<0$ and hence $g(a)\ge 1$. 
If $a\ge 2$, then 
\begin{equation}\label{eqg}
g(a) = \frac{a^{1-q}-1}{a-1} \ge \frac{a^{1-q}-a^{1-q}/2}{a-1} = \frac12\frac{a^{1-q}}{a-1}. 
\end{equation}
Thus, for $v(x,t)>2v(y,t)$, we may combine \eqref{eqDg} and \eqref{eqg} to obtain 
\begin{align}
&D  \ge K(x,y,t)\frac{\phi^p(x,t)}{v^{q}(y,t)}(v(x,t)-v(y,t))^p\left(\frac12\frac{\frac{v^{q-1}(y,t)}{v^{q-1}(x,t)}}{\frac{v(x,t)}{v(y,t)}-1}-c_p\delta\frac{v(y,t)}{v(x,t)}\right)\\
& =  K(x,y,t)\frac{\phi^p(x,t)}{v^{q}(y,t)}(v(x,t)-v(y,t))^p\left(\frac12\frac{\frac{v^{q}(y,t)}{v^{q-1}(x,t)}}{v(x,t)-v(y,t)}-c_p\delta\frac{v(y,t)}{v(x,t)}\right)\notag\\
& = K(x,y,t)\frac{(v(x,t)-v(y,t))^{p-1}}{v^{q-1}(x,t)}\phi^p(x,t)\notag\\
&\times\left(\frac12 - c_p\delta \frac{v(y,t)}{v(x,t)}\frac{v(x,t)-v(y,t)}{v^q(y,t)}v^{q-1}(x,t)\right).\notag
\end{align}
Recalling that $v(x,t)\ge 2v(y,t)$, $q<0$ and $v(y,t)>0$ since $y\in B_r$, we see that 
\[
-\frac{v(y,t)}{v(x,t)}\frac{v(x,t)-v(y,t)}{v^q(y,t)}v^{q-1}(x,t) \ge -\frac{1}{2}.  
\]
Thus
\begin{equation}
D \ge K(x,y,t)\frac{(v(x,t)-v(y,t))^{p-1}}{v^{q-1}(x,t)}\phi^p(x,t)\left(\frac12-\frac12c_p\delta\right). 
\end{equation}
At this point we observe that 
\begin{align}
&\frac{(v(x,t)-v(y,t))^{p-1}}{v^{q-1}(x,t)} \ge 2^{1-p}v^{p-q}(x,t) \ge 2^{1-p}(v^{\frac{p-q}{p}}(x,t) - v^{\frac{p-q}{p}}(y,t))^p.
\end{align}
Choosing 
\begin{equation}\label{eqdelta}
\delta = \frac{1}{4c_p}, 
\end{equation}
 we arrive at 
\begin{align}\label{xge2y}
D &\ge K(x,y,t)2^{-p-1}(v^{\frac{p-q}{p}}(x,t) - v^{\frac{p-q}{p}}(y,t))^p\phi^p(x,t)\\
& = K(x,y,t)2^{-p-1}(w(x,t)-w(y,t))^p\phi^p(x,t). \notag
\end{align}

We now consider the remaining case $v(y,t)<v(x,t)<2v(y,t)$. By \eqref{eqDg}, the fact that 
$g(a)\ge 1$ and the choice of $\delta$, we have 
\begin{equation}\label{xle2y2}
D  \ge \frac12K(x,y,t)\frac{\phi^p(x,t)}{v^q(x,t)}(v(x,t)-v(y,t))^p.
\end{equation}
We further estimate 
\begin{align}\label{xle2y}
(w(x,t)-w(y,t))^p &= \left(\frac{p}{p-q}\right)^p\left(\int_{v(y,t)}^{v(x,t)}\tau^{-q/p}d\tau\right)^p\\
&\le \left(\frac{p}{p-q}\right)^p \frac{((v(x,t)-v(y,t))^p}{v^q(y,t)}.\notag 
\end{align}
Combining \eqref{xge2y}, \eqref{xle2y2} and \eqref{xle2y}, we have shown that 
\begin{equation}\label{eqD}
D \ge 2^{-1-p}K(x,y,t)(w(x,t)-w(y,t))^p\phi^p(x,t). 
\end{equation}
For the estimate of $E$ we use the facts that 
\[
-v^{1-q}(y,t) \ge -v^{1-q}(x,t)\text{ and }(v(x,t)-v(y,t))/v(x,t)\le 1,
\]
to find that 
\begin{equation}\label{eqE}
E \ge -CK(x,y,t)w^p(x,t)|\phi(x,t)-\phi(y,t)|^p.
\end{equation}
Finally, combining \eqref{eqD} and \eqref{eqE}, we have shown that for $v(x,t)>v(y,t)$, 
\begin{equation}\label{DE}
D+E \ge CK(x,y,t)((w(x,t)-w(y,t))^p\phi^p(x,t)-w^p(x,t)|\phi(x,t)-\phi(y,t)|^p). 
\end{equation}
If $v(y,t)>v(x,t)$, the same estimate may be deduced by interchanging the roles of $x$ and $y$. If $v(x,t) = v(y,t)$ it is sufficient to note that $0 \ge E$. 
Using the fact that 
\begin{align*}
&|w(x,t)\phi(x,t)-w(y,t)\phi(y,t)| - c\max\{w^p(x,t),w^p(y,t)\}|-\phi(x,t)\phi(y,t)|^p\\
& \le c|w(x,t)-w(y,t)|^p\phi^p(x,t),  
\end{align*}
and recalling that we are actually dealing with $v_m$ rather than $v$, we have shown that 
\begin{align}\label{I1final}
I_1 & \ge c\int_{t_1}^{t_2}\int_{\R^n\setminus B_r}\int_{B_r}\frac{|w_m(x,t)\phi(x,t)-w_m(y,t)\phi(y,t)|}{|x-y|^{n+sp}}dxdydt\\
&-c\int_{t_1}^{t_2}\int_{\R^n\setminus B_r}\int_{B_r}\max\{w_m^p(x,t),w_m^p(y,t)\}\frac{|\phi(x,t)-\phi(y,t)|^p}{|x-y|^{n+sp}}dxdydt,\notag
\end{align}
where we have set $w_m = v_m^{p-1+\xi}$. 

We now turn to $I_2$, and first observe that 
\begin{align}\label{I2ygex}
I_2 & \ge 2\int_{t_1}^{t_2}\int_{\R^n\setminus B_r}\int_{B_r}\A v_m(x,y,t)\chi_{\{v(y,t)>v(x,t)\}}\eta(x,t)dxdydt.
\end{align} 
We will need the following inequality to estimate $I_2$:
If $0\le a<b$, then  
\begin{equation}\label{aleb}
|a|^{p-2}a-|b|^{p-2}b \le |a-b|^{p-2}(a-b). 
\end{equation}
To prove \eqref{aleb}, we make use of the fact that the $l^s$- norm of an element $(\alpha,\beta)$ of $\R^2$ is non-increasing in $s$:  
\[
|(\alpha,\beta)|_s\le |(\alpha,\beta)|_1,\quad s\ge 1. 
\]
If $\alpha,\beta>0$, this means that 
\begin{equation}\label{alphabeta}
\alpha^s+\beta^s\le (\alpha+\beta)^s,\quad s\ge 1. 
\end{equation}
Now \eqref{aleb} follows by taking $\alpha = a$, $\beta = b-a$ and $s=p-1$ in \eqref{alphabeta}. 
Using \eqref{aleb} in \eqref{I2ygex} gives
\begin{align}\label{I2final}
I_2 &\ge c \int_{t_1}^{t_2}\int_{\R^n\setminus B_r}\int_{B_r}\frac{|v(x,t)|^{p-2}v(x,t)}{|x-y|^{n+sp}}\chi_{\{v(y,t)>v(x,t)\}}\eta(x,t)dxdydt\\
&-c \int_{t_1}^{t_2}\int_{\R^n\setminus B_r}\int_{B_r}\frac{|v(y,t)|^{p-2}v(y,t)}{|x-y|^{n+sp}}\chi_{\{v(y,t)>v(x,t)\}}\eta(x,t)dxdydt\notag\\
& \ge - c\int_{t_1}^{t_2}\int_{\R^n\setminus B_r}\int_{B_r}\frac{v(y,t)_+^{p-1}}{|x-y|^{n+sp}}\eta(x,t)dxdydt\notag\\
& \ge -c2^{p-1}\int_{t_1}^{t_2}\int_{\R^n\setminus B_r}\int_{B_r}\frac{u(y,t)_+^{p-1}}{|x-y|^{n+sp}}\eta(x,t)dxdydt\notag\\
& -c2^{p-1}\int_{t_1}^{t_2}\int_{\R^n\setminus B_r}\int_{B_r}\frac{d^{p-1}}{|x-y|^{n+sp}}\eta(x,t)dxdydt\notag\\
& \ge -c\int_{t_1}^{t_2}\sup_{x\in\supp\psi}\int_{\R^n\setminus B_r}\frac{u(y,t)_+^{p-1}}{|x-y|^{n+sp}}dy\int_{B_r}v_m^{\xi}\phi^p(x,t)dxdt\notag\\
&-c\sup_{x\in\supp\psi}\int_{\R^n\setminus B_r}\frac{1}{|x-y|^{n+sp}}dy\int_{t_1}^{t_2}\int_{B_r}v_m^{p-1+\xi}\phi^p(x,t)dxdt\notag
\end{align}
where we used the fact that $d\le v_m$ in $B_r\times(t_1,t_2)$.

Recalling \eqref{subsolnoh} and collecting the estimates \eqref{I1final} and  \eqref{I2final} for $I_1$ and $I_2$ respectively, we arrive at 
\begin{align}
&C\int_{t_1}^{t_2}\int_{\R^n\setminus B_r}\int_{B_r}\frac{|w_m(x,t)\phi(x,t)-w_m(y,t)\phi(y,t)|}{|x-y|^{n+sp}}dxdydt\\
&+\frac{1}{1+\xi}\int_{B_r}v_m^{1+\xi}(x,t)\phi^p(x,t)dx\bigg |_{t=t_1}^{t_2}\notag\\
&\le C\int_{t_1}^{t_2}\int_{\R^n\setminus B_r}\int_{B_r}\max\{w_m^p(x,t),w_m^p(y,t)\}\frac{|\phi(x,t)-\phi(y,t)|^p}{|x-y|^{n+sp}}dxdydt\notag\\
&+C\sup_{x\in\supp\psi}\int_{\R^n\setminus B_r}\frac{1}{|x-y|^{n+sp}}dy\int_{t_1}^{t_2}\int_{B_r}v_m^{p-1+\xi}\phi^p(x,t)dxdt\notag\\
&+ C\int_{t_1}^{t_2}\sup_{x\in\supp\psi}\int_{\R^n\setminus B_r}\frac{u(y,t)_+^{p-1}}{|x-y|^{n+sp}}dy\int_{B_r}v_m^{\xi}\phi^p(x,t)dxdt\notag\\
& +\frac{1}{1+\xi}\int_{t_1}^{t_2}\int_{B_r}v_m^{1+\xi}\partial_t\phi^pdxdt.\notag
\end{align}
Passing to the limit $m\to\infty$, we obtain the conclusion of the lemma.

\end{proof}

\begin{lemma} \label{cor_cacc}
Let $p\in (1,\infty)$, $s\in (0,1)$. Let $\xi \ge 1$ and assume that $K$ satisfies the ellipticity condition \eqref{ellipticity}. Let $x_0\in\mathbb R^n$, $\tau_1<\tau_2$, $B_r:=B_r(x_0)$, and assume that $u $ is a non-negative sub-solution in $B_r\times (\tau_1, \tau_2)$. For $d>0$, let $v=u+d$, $w=v^{(p-1+\xi)/p}$. Then
	\begin{align*}
		&\int_{\tau_1}^{\tau_2} \int_{B_r} \int_{B_r}  {|w\phi(x,t)-w\phi(y,t)|^{p}}\, d\mu + \frac{1}{\xi+1} \sup_{\tau_1<t<\tau_2}\int_{B_r}
v(x,t)^{1+\xi} \phi^p(x,t) dx \\
&\leq C \int_{\tau_1}^{\tau_2} \int_{B_r} \int_{B_r} \max\{w(x,t), w(y,t)\}^p|\phi(x,t)-\phi(y,t)|^p\, d\bar\mu\notag\\
&+C\biggl (\sup_{x\in\supp\psi}\int_{\R^n\setminus B_r}|x-y|^{-(n+ps)}\, dy\biggr)\biggl (\int_{\tau_1}^{\tau_2}\int_{B_r}w^p(x,t)\phi^p(x,t)\, dx dt\biggr )\\
&+C\int_{\tau_1}^{\tau_2}\left(\sup_{x\in\supp\psi}\int_{\R^n\setminus B_r}\frac{u(y,t)_+^{p-1}}{|x-y|^{n+sp}}dy\int_{B_r}v^\xi\phi^p(x,t)dx\right)dt\\
		&+ \frac{1}{(1+\xi)}\int_{\tau_1}^{\tau_2} \int_{B_r}v^{1+\xi} \left (  \frac{\partial \phi^p}{\partial t} \right )_{+} dx dt,
	\end{align*}
	for all $\phi(x,t) = \psi(x) \zeta(t)$ with $\zeta \in C_0^{\infty}(\tau_1,\tau_2)$ and $\psi \in C_0^{\infty}(B_r)$.
	\end{lemma}

\begin{proof}
We proceed as in the proof of Lemma \ref{lem_cacc} but leave out $\theta_j$ from the test function, i.e.\ we use the test function
\[
\eta_{h} = \left((v_m)_h^{1-q}\phi^p\right)_h. 
\]
This leads to the desired estimate, save for the term 
\[
\frac{1}{\xi+1} \sup_{\tau_1<t<\tau_2}\int_{B_r}
v(x,t)^{1+\xi} \phi^p(x,t) dx
\]
on the left hand side. For any given $\ep>0$, we may choose $t_2 = t_2(\ep)\in (\tau_1,\tau_2)$ in Lemma \ref{lem_cacc} so that 
\begin{align*}
&\frac{1}{\xi+1} \sup_{\tau_1<t<\tau_2}\int_{B_r}
v(x,t)^{1+\xi} \phi^p(x,t) dx\\
&\le\frac{1}{\xi+1} \int_{B_r}
v(x,t_2)^{1+\xi} \phi^p(x,t) dx + \ep.
\end{align*}
Then, choosing $t_1\in (\tau_1,t_2)$ outside the support of $\zeta$
and letting $\ep\to0$, we obtain the conclusion of the lemma. 

\end{proof}

\subsection{Local boundedness of subsolutions}

Based upon the parabolic Sobolev inequality \eqref{parabSobolev} and Lemma \ref{cor_cacc} $(ii)$, we are able to prove a reverse H\"older inequality for subsolutions and do a Moser iteration to prove local boundedness. 
The following standard lemma, see e.g.\ \cite{HanLin} Lemma 4.3, is used in the proof.
\begin{lemma}\label{HL}
Suppose $f(s)$ is nonnegative and bounded in $[0,1]$. If for all $0\le \alpha<\beta\le 1$,
\[
f(\alpha) \le \frac12 f(\beta) + \frac{A}{(\alpha-\beta)^\gamma} + B,
\]
then 
\[
f(\alpha) \le c(\gamma)\left(\frac{A}{(\alpha-\beta)^\gamma} + B\right).
\]
\end{lemma}

\begin{lemma}\label{Moser}

Let $0<r<R$ and let $Q = B_r\times(t_0-T_0,t_0)$. Suppose that $u$ is a nonnegative subsolution in $2Q$. 
Let $v(x,t) = u(x,t) + d$, where 
\begin{equation}\label{eqd}
d = \tail_\infty(u_+;x_0,\sigma r,t_0-T_0,t_0) + \left(\frac{r^{sp}}{T_0}\right)^{\frac{1}{p-2}}. \notag
\end{equation}
Then for any $\sigma\in (0,1)$,   
\begin{equation}
\sup_{B(x_0,\sigma r)\times(t_0-\sigma^{sp}T_0,t_0)}v\le \left(\frac{T_0}{r^{sp}}\frac{C}{(1-\sigma)^{\alpha}}\fint_Qv^{p-2+\delta}dxdt\right)^{\frac{1}{\delta}}, 
\end{equation}
 where $\alpha = \frac{(n+sp)^2}{sp}$. 
\end{lemma}

\begin{proof}
Let $\sigma\in(0,1)$ and let $\theta = 1-\sigma\in (0,1)$. We set
	\begin{equation*}
		r_0 = r, \quad r_j = r-\theta r(1-2^{-j}), \quad \delta_j = 2^{-j}\theta r,\quad j=1,2,\ldots
	\end{equation*}
	and
	\begin{eqnarray*}
		U_j &=& B_j \times \Gamma_j = B(x_0,r_j) \times (t_0-(r_j/r)^{sp} T_0, t_0), \\
		U(\lambda) &=& B(\lambda) \times \Gamma(\lambda) = B(x_0,\lambda r) \times (t_0 - \lambda^{sp}T_0, t_0),\quad \lambda>0.
	\end{eqnarray*}
	We choose test functions $\psi_{j} \in C^\infty(B_j)$ and $\zeta_j \in C^{\infty}(\Gamma_j)$ satisfying 
	\begin{equation}\label{deltaspace}
	\psi_j\equiv 1\text{ in }B_{j+1},\;\text{dist}( \supp \psi_j,\R^n\setminus B_j) \ge \frac{\delta_j}{2},
	\end{equation}	
 such that for $\phi_j = \psi_j \zeta_j$ we have
	\begin{equation*}
		0 \leq \phi_j \leq 1, \quad \phi_j=1 \text{ in } U_{j+1},\quad \phi_j = 0 \text{ on } \partial_p U_j,
	\end{equation*}
	and
	\begin{equation} \label{eqphi}
		|\nabla\phi_j| \leq \frac{C}{\theta r}2^j = C\delta_j^{-1}, \quad \left| \frac{\partial \phi_j}{\partial t}\right| \leq \frac{r^{sp}}{T_0} \frac{C}{(\theta r)^{sp}}2^{spj} = \frac{r^{sp}}{T_0}C\delta_j^{-sp}.
	\end{equation}
Let 
\[
w = v^{\frac{p-1+\xi}{p}}\quad \text{and }\eta_j = \phi_j^{\frac{p-1+\xi}{\xi+1}}.
\]
Note that $\eta_j$ satisfies the same bounds \eqref{eqphi} as $\phi_j$ with $C = C(n)\frac{p-1+\xi}{\xi+1}$. 
By the Sobolev embedding theorem there holds,  
\begin{align}\label{sobolevUj}
&\int_{\Gamma_j}\fint_{B_j}|w\eta_j|^{\kappa p}dxdt\\
&\quad\le Cr_j^{sp-n}\int_{\Gamma_j}\int_{B_{j}}\int_{B_{j}}\frac{|w\eta_j(x,t)-w\eta_j(y,t)|^p}{|x-y|^{n+sp}}dxdydt	\notag\\
&\quad\times\left(\sup_{\Gamma_j}\fint_{B_j}|w\eta_j|^{pG(\kappa-1)}\right)^{\frac{1}{G}} = Cr_j^{sp-n}I_1\times \left(\frac{I_2}{|B_j|}\right)^{\frac{1}{G}},\notag 
\end{align}
where $G=\frac{\kappa^*}{\kappa^*-1}$. 
By Lemma \ref{cor_cacc}, 
\begin{align}\label{I1}
&I_1 + \frac{1}{1+\xi}\sup_{t\in\Gamma_j}\int_{B_j}v^{1+\xi}\eta_j(x,t)dx\\
&\le  C \int_{\Gamma_j}\int_{B_j} \int_{B_j} \max\{w(x,t), w(y,t)\}^p|\eta_j(x,t)-\eta_j(y,t)|^p\, d\bar\mu\notag\\
&+C\biggl (\sup_{x\in\supp\psi_j}\int_{\R^n\setminus B_j}|x-y|^{-(n+ps)}\, dy\biggr)\biggl (\int_{\Gamma_j}\int_{B_j}w^p(x,t)\eta_j^p(x,t)\, dx dt\biggr )\notag\\
& + C\sup_{x\in\supp\psi}\int_{t_1}^{t_2}\int_{\R^n\setminus B_r}\frac{u(y,t)_+^{p-1}}{|x-y|^{n+sp}}dy\int_{B_r}v^{\xi}\eta_j^p(x,t)dxdt\notag\\
		&+ C\int_{\Gamma_j} \int_{B_j}v^{1+\xi}\eta_j \left ( \frac{1}{1+\xi} \frac{\partial \eta_j^p}{\partial t} \right )_{+} dx dt
= I_{11} + I_{12}+I_{13} + I_{14}.\notag
\end{align}
For $I_{11}$ we have the estimate 
\[
I_{11} \le C\int_{\Gamma_j}\int_{B_j}  w^p(x,t)dxdt\sup_{x\in B_j}\int_{B_j}\frac{|\eta_j(x,t)-\eta_j(y,t)|^p}{|x-y|^{n+sp}}dy. 
\]
Using \eqref{eqphi}, we see that for any $x\in B_j$, 
\begin{align*}
&\int_{B_j}\frac{|\eta_j(x,t)-\eta_j(y,t)|^p}{|x-y|^{n+sp}}dy\\
&\le \int_{B_j\cap \{|x-y|\le \delta_j\}}\frac{|\eta_j(x,t)-\eta_j(y,t)|^p}{|x-y|^{n+sp}}dy + \int_{B_j\cap \{|x-y|> \delta_j\}}\frac{|\eta_j(x,t)-\eta_j(y,t)|^p}{|x-y|^{n+sp}}dy\\
&\le C\delta_j^{-p}\int_{B_j\cap \{|x-y|\le \delta_j\}}\frac{|x-y|^p}{|x-y|^{n+sp}}dy + \int_{B_j\cap \{|x-y|> \delta_j\}}\frac{2^p}{|x-y|^{n+sp}}dy\\
&\le C\delta_j^{-sp}. \end{align*}
Since 
\[
\delta_j^{-sp} = r_j^{-sp}\frac{r_j^{sp}}{\delta_j^{sp}}\le r_j^{-sp}2^{jsp}\left(\frac{r}{(1-\sigma)r}\right)^{sp} = r_j^{-sp}\frac{2^{jsp}}{\theta^{sp}},
\]
we get 
\begin{equation}\label{I11}
I_{11} \le Cr_j^{-sp}\frac{2^{jsp}}{\theta^{sp}}\int_{U_j}  w^p(x,t)dxdt. 
\end{equation}
The first factor of $I_{12}$ can be estimated by $C(\delta_j^{-sp}+r_j^{-sp})$, using \eqref{deltaspace} and polar coordinates. 
This gives us 
\begin{equation}\label{I12}
I_{12}\le C(\delta_j^{-sp}+r_j^{-sp})\int_{U_j}  w^p(x,t)dxdt \le Cr_j^{-sp}\frac{2^{jsp}}{\theta^{sp}}\int_{U_j}  w^p(x,t)dxdt.
\end{equation} 
 We now turn to $I_{13}$ and first note that if $y\in\R^n\setminus B_{j}$ and $x\in \sup \psi_j$, then 
\begin{align*}
\frac{1}{|x-y|} & = \frac{1}{|x_0-y|}\frac{|x_0-y|}{|x-y|}\le \frac{1}{|x_0-y|}\frac{|x-x_0| + |x-y|}{|x-y|}\\
&\le \frac{1+2r_j/\delta_j}{|x_0-y|} \le \frac{C\theta^{-1} 2^j}{|x_0-y|}. 
\end{align*}
It follows that 
\begin{align}
I_{13} & \le \frac{C}{\theta^{n+sp}}2^{j(n+sp)}r_j^{-sp}\tail_\infty^{p-1}(x_0;r_j,\Gamma_j)\int_{U_j}v^\xi dxdt\\
& \le \frac{C}{\sigma^{sp}\theta^{n+sp}}2^{j(n+sp)}r_j^{-sp}\tail_\infty^{p-1}(x_0;\sigma r,\Gamma_0)\int_{U_j}v^\xi dxdt\notag\\
&\le \frac{C}{\sigma^{sp}\theta^{n+sp}}2^{j(n+sp)}r_j^{-sp}\frac{\tail_\infty^{p-1}(x_0;\sigma r,\Gamma_0)}{d^{p-1}}\int_{U_j}v^{\xi+p-1} dxdt\notag\\
&\le \frac{C}{\sigma^{sp}\theta^{n+sp}}2^{j(n+sp)}r_j^{-sp}\int_{U_j}w^{p} dxdt.\notag
\end{align}

Finally, by applying \eqref{eqd} and \eqref{eqphi}, we obtain the following estimate for $I_{14}$:
\begin{align}\label{I14}
I_{14} & \le C\int_{\Gamma_j}\int_{B_j}v^{1+\xi}\frac{r^{sp}}{T_0}\delta_j^{-sp} \\
& \le C\int_{\Gamma_j}\int_{B_j}\frac{w^pT_0}{r^{sp}}\frac{r^{sp}}{T_0}\delta_j^{-sp}dxdt
\le C\delta_j^{-sp} \int_{\Gamma_j}\int_{B_j}w^pdxdt \notag \\
&\le Cr_j^{-sp}\frac{2^{jsp}}{\theta^{sp}}\int_{\Gamma_j}\int_{B_j}w^pdxdt.\notag
\end{align}
Putting together \eqref{I1} with \eqref{I11} - \eqref{I14} yields 
\begin{align}\label{I1final}
&I_1 + \frac{1}{1+\xi}\sup_{t\in\Gamma_j}\int_{B_j}v^{1+\xi}\eta_j(x,t)dx \\
&\le C\frac{2^{j(n+sp)}}{\sigma^{sp}\theta^{n+sp}}r_j^{-sp}\int_{\Gamma_j}\int_{B_j}w^pdxdt = r_j^{-sp}c_{j,\theta}\int_{U_j}w^pdxdt, \notag
\end{align}
where $c_{j,\theta}=C\frac{2^{j(n+sp)}}{(1-\theta)^{sp}\theta^{n+sp}}$.
Let 
\begin{equation}\label{kappa}
\kappa = 1 + \frac{1+\xi}{G(p-1+\xi)}. 
\end{equation}
Note that $\kappa\in(1,\kappa^*)$. Then
\begin{align}\label{I2}
I_2 & =  \sup_{\Gamma_j}\int_{B_j}v^{1+\xi}\phi_j^pdx. 
\end{align}
In view of \eqref{I1final}, $\frac{I_2}{1+\xi}$ enjoys the same estimate as $I_1$ with $\phi_j$ in place of $\eta_j$, i.e.
\begin{equation}
I_2 \le (1+\xi)c_{j,\theta}r_j^{-sp}\int_{U_j}w^pdxdt.
\end{equation}
At this point we have shown, recalling \eqref{sobolevUj}, that 
\begin{align}\label{revhold1}
&\fint_{U_j}|w\phi_j|^{\kappa p}dxdt \\
&\le Cr_j^{sp}\left(r_j^{-sp}c_{j,\theta}\fint_{\Gamma_j}\fint_{B_j}w^pdxdt\right)\left(r_j^{-sp}c_{j,\theta}(1+\xi)\int_{\Gamma_j}\fint_{B_j}w^pdxdt\right)^{\frac{1}{G}}\notag\\
&\le Cr_j^{sp}|\Gamma_j|^{\frac{1}{G}}(1+\xi)^{\frac{1}{G}}\left(r_j^{-sp}c_{j,\theta}\fint_{U_j}w^pdxdt\right)^{1+\frac{1}{G}}.\notag
\end{align}
Let $\gamma = 1+1/G = (n+sp)/n$. Then
\begin{align}\label{revhold2}
&\fint_{U_{j+1}}|w|^{\kappa p}dxdt \le C\frac{|U_j|}{|U_{j+1}|}r_j^{sp}|\Gamma_j|^{\frac{1}{G}}(1+\xi)^{\frac{1}{G}}\left(r_j^{-sp}c_{j,\sigma}\fint_{U_j}w^pdxdt\right)^{\gamma}\\
& = C\frac{r_j^{n+sp}}{r_{j+1}^{n+sp}}r_j^{sp\gamma}(1+\xi)^{\frac{1}{G}}\left(\frac{T_0}{r^{sp}}\right)^{\frac{1}{G}}\left(r_j^{-sp}c_{j,\sigma}\fint_{U_j}w^pdxdt\right)^{\gamma}\notag \\
& = C\left(\left(\frac{T_0}{r^{sp}}\right)^{\frac{\gamma-1}{\gamma}}(1+\xi)^{\frac{\gamma-1}{\gamma}}c_{j,\sigma}\fint_{U_j}w^pdxdt\right)^{\gamma}.\notag 
\end{align}

Recalling the definitions of $w$, $\kappa$ and $G$ we may rewrite \eqref{revhold2} as 
\begin{align}
\label{revhold3}
&\fint_{U_{j+1}}|v|^{p-1+\frac{sp}{n} +\gamma\xi}dxdt \\
& \le C\left(\left(\frac{T_0}{r^{sp}}\right)^{\frac{\gamma-1}{\gamma}}(1+\xi)^{\frac{\gamma-1}{\gamma}}c_{j,\sigma}\fint_{U_j}|v|^{p-1+\xi}dxdt\right)^{\gamma}.\notag
\end{align}
We are now in a position to start a Moser iteration.
Fix $\xi_0 >1$ and set
\begin{align*}
&\xi_{j} = \gamma^j(\xi_0+1)-1,\quad j\ge 0,\\
&p_j = p-1+\xi_j,\quad j\ge 0. 
\end{align*} 
Then we have the inductive relations 
\begin{align}
&\xi_{j+1} = \gamma(\xi_j+1)-1,\\
&p-1+\frac{sp}{n} +\gamma\xi_j = p-1+\xi_{j+1} = p_{j+1}.
\end{align}
Hence, using $\xi = \xi_j$ in \eqref{revhold3} and estimating $(1+\xi_j)^{\frac{\gamma-1}{\gamma}}\le 2\xi_0^{\frac{\gamma-1}{\gamma}}\gamma^j$, we find that 
\begin{align}
\label{revhold4}
&\fint_{U_{j+1}}|v|^{p_{j+1}}dxdt \le C\left(\left(\frac{T_0}{r^{sp}}\right)^{\frac{\gamma-1}{\gamma}}\xi_0^{\frac{\gamma-1}{\gamma}}\gamma^jc_{j,\sigma}\fint_{U_j}|v|^{p_j}dxdt\right)^{\gamma},
\end{align}
for $j=0,1,\ldots$. 
By iterating \eqref{revhold4} $m$ times, starting at $p_m$ and taking  $\gamma^m$:th roots, we conclude the estimate
\begin{align}\label{miter}
& \left(\fint_{U_{m}}|v|^{p_{m}}dxdt\right)^{\frac{1}{\gamma^m}} \\
&\le \fint_{U(r)}|v|^{p_0}\prod_{j=1}^mC^{\gamma^{1-j}}c_{j,\theta}^{\gamma^{-j}}\left(\frac{T_0}{r^{sp}}\right)^{\frac{\gamma-1}{\gamma}\gamma^{-j}}\xi_0^{\frac{\gamma-1}{\gamma}\gamma^{-j}}(\gamma^j)^{\gamma^{-j}}. 
\end{align}
The limit as $m\to\infty$ of the product on the right hand side of \eqref{miter} may be estimated in a standard fashion by studying its logarithm. 
Thus we obtain 
\begin{equation}
\lim_{m\to\infty}\prod_{j=1}^mC^{\gamma^{1-j}}c_{j,\theta}^{\gamma^{-j}}\left(\frac{T_0}{r^{sp}}\right)^{\frac{\gamma-1}{\gamma}\gamma^{-j}} \le C\theta^{-\frac{(n+sp)^2}{sp}}\frac{T_0}{r^{sp}}\xi_0. 
\end{equation}
Since $\lim_{m\to\infty}\frac{p_m}{\gamma^m} = \xi_0+1$, taking $m\to\infty$ in \eqref{revhold4} results in 
\begin{align}
\sup_{U(\sigma)}v^{\xi_0+1}\le C\xi_0(1-\sigma)^{-\frac{(n+sp)^2}{sp}}\frac{T_0}{r^{sp}}\fint_{U(1)}v^{p-1+\xi_0}dxdt.
\end{align}
Additionally, 
\begin{align}
&\sup_{U(\sigma)}v^{\xi_0+1}\le \sup_{U(1)}v^2C\xi_0(1-\sigma)^{-\frac{(n+sp)^2}{sp}}\frac{T_0}{r^{sp}}\fint_{U(1)}v^{p-3+\xi_0}dxdt\\
&\le \frac12\sup_{U(1)}v^{\xi_0+1} + C\xi_0\left((1-\sigma)^{-\frac{(n+sp)^2}{sp}}\frac{T_0}{r^{sp}}\fint_{U(1)}v^{p-3+\xi_0}dxdt\right)^{\frac{\xi_0+1}{\xi_0-1}}, 
\end{align}
where we used Young's inequality with exponents $\frac{\xi_0+1}{2}$ and $\frac{\xi_0+1}{\xi_0-1}$. By Lemma \ref{HL} we have 
\begin{align}
&\sup_{U(\sigma r)}v^{\xi_0+1}\le C\left((1-\sigma)^{-\frac{(n+sp)^2}{sp}}\frac{T_0}{r^{sp}}\fint_{U(r)}v^{p-3+\xi_0}dxdt\right)^{\frac{\xi_0+1}{\xi_0-1}}.
\end{align}
Whence the result follows by choosing $\xi_0 = 1+\delta$.


\end{proof}

In the next lemma we extract information on $u$ from Lemma \ref{Moser}.
\begin{lemma}\label{sup_param}
Let $u$ and $Q$ be as in Lemma \ref{Moser}. 
Then for any $\sigma\in(0,1)$,
\begin{align}
\sup_{\sigma Q}u &\le \frac{C}{(1-\sigma)^\alpha}\left(\left(\frac{r^{sp}}{T_0}\right)^{\frac{1}{p-2}}  + \frac{T_0}{r^{sp}}\tail_\infty^{p-1}(u_+;x_0,\sigma r,t_0-T_0,t_0)\right)\notag\\
&+\frac{C}{(1-\sigma)^\alpha}\frac{T_0}{r^{sp}}\fint_{Q}u^{p-1}(x,t)dxdt,\notag
\end{align}
where $\alpha = \frac{(n+sp)^2}{sp}$. 
\end{lemma}

\begin{proof}
Choosing $\delta= 1$ in Lemma \ref{Moser}, we get
\begin{align}
&\sup_{\sigma Q}u\le \sup_{\sigma Q}v \le \frac{C}{(1-\sigma)^\alpha}\frac{T_0}{r^{sp}}\left(\frac{r^{sp}}{T_0}\right)^{\frac{p-1}{p-2}}\\
&+ \frac{C}{(1-\sigma)^\alpha}\frac{T_0}{r^{sp}}\left(\tail_\infty^{p-1}(u_+;x_0,\sigma r,t_0-T_0,t_0) + \fint_{Q}u^{p-1}(x,t)dxdt\right),\notag  
\end{align}
from which the claim easily follows.  
\end{proof}

%
%

\begin{proof}[\bf Proof of Theorem \ref{main_thm}]
For any $\ep\in (0,1)$, let $r = \ep R$ and let $T_1 = \ep^{sp}T_0$, so that \[\ep Q = B_{r}\times(t_0-T_1,t_0).\] Let $\psi \in C^\infty(B_{r})$ and $\zeta \in C^{\infty}(t_0-T_0,t_0)$ satisfy 
	\begin{equation}\label{deltaspace}
	\psi\equiv 1\text{ in }B_{r},\quad\text{dist}( \supp \psi,\R^n\setminus B_R) \ge \frac{(R-r)}{2}=\frac{\ep R}{2}=:\delta/2,
	\end{equation}	
and be such that for $\phi = \psi \zeta$, we have
	\begin{equation*}
		0 \leq \phi \leq 1, \quad \phi=1 \text{ in } \ep Q,\quad \phi = 0 \text{ on } \partial_p Q,
	\end{equation*}
	and
	\begin{equation} \label{eqphi2}
		|\nabla\phi| \leq \frac{C}{\ep r} = C\delta^{-1}, \quad \left| \frac{\partial \phi}{\partial t}\right| \leq \frac{R^{sp}}{T_0} \frac{C}{(\ep r)^{sp}} = \frac{R^{sp}}{T_0}C\delta^{-sp}.
	\end{equation}
    Let 
    \[
    v = u + \left(\frac{R^{sp}}{T_0}\right)^{\frac{1}{p-2}} + \tail_\infty(u_+;x_0,\ep r,t_0-T_0,t_0). 
    \]
By the parabolic Sobolev embedding theorem \ref{parabSobolev}, with $f = v\phi$ and $\kappa = 1+s/n$, we have 
\begin{align}
\int_{t_0-T_1}^{t_1}\int_{B_r}&|v|^{p+\frac{sp}{n}}dxdt = \int_{\ep Q}|v|^{p+\frac{sp}{n}}dxdt\\
&\le CR^{sp}\int_{t_0-T_0}^{t_0}\int_{B_R}\int_{B_R}\frac{|v\phi(x,t)-v\phi(y,t)|^p}{|x-y|^{n+sp}}dxdydt\notag\\
&\quad\times\left(\sup_{t_0-T_0<t<t_0}\fint_{B_R}|v(x,t)|dx\right)^{\frac{sp}{n}}.\notag
\end{align}
The term 
\[
R^{sp}\int_{t_0-T_0}^{t_0}\int_{B_R}\int_{B_R}\frac{|v\phi(x,t)-v\phi(y,t)|^p}{|x-y|^{n+sp}}dxdydt
\]
may be estimated precisely the way we treated the term $I_1$ in Lemma \ref{Moser}, using the Cacciopollo inequality with $\xi=1$. We thus end up with 
\begin{align}
\int_{\ep Q}&|v|^{p+\frac{sp}{n}}dxdt\\
&\le \frac{C}{(1-\ep)^{n+sp}}\int_Q|v|^pdxdt\left(\sup_{t_0-T_0<t<t_0}\fint_{B_R}|v(x,t)|dx\right)^{\frac{sp}{n}}.\notag
\end{align}
Using H\"older's inequality followed by Young's inequality, we get
\begin{align}
&\int_{\ep Q}|v|^pdxdt\le \left(\int_{\ep Q}|v|^{p+\frac{sp}{n}}dxdt\right)^{\frac{n}{n+s}}\\
&\le \left(\frac{C}{(1-\ep)^{n+sp}}\int_Q|v|^pdxdt\left(\sup_{t_0-T_0<t<t_0}\fint_{B_R}|v(x,t)|dx\right)^{\frac{sp}{n}}\right)^{\frac{n}{n+s}}\notag\\
&\le \frac12\int_Q|v|^pdxdt + \frac{C}{(1-\ep)^{\frac{(n+sp)n}{s}}}\left(\sup_{t_0-T_0<t<t_0}\fint_{B_R}|v(x,t)|dx\right)^p.\notag
\end{align}
An application of Lemma \ref{HL} gives 
\begin{align}\label{anothereq}
&\int_{\ep Q}|v|^pdxdt\le \frac{C}{(1-\ep)^{\frac{(n+sp)n}{s}}}\left(\sup_{t_0-T_0<t<t_0}\fint_{B_R}|v(x,t)|dx\right)^p.
\end{align}
Now, Lemma \ref{Moser} with $\delta = 1$ and $\sigma =\ep$, in conjunction with H\"older's inequality and \eqref{anothereq}, gives,  
\begin{align}
& \sup_{\ep^2 Q}v \le \frac{T_1}{r^{sp}}\frac{C}{(1-\ep)^{\frac{(n+sp)^2}{sp}}}\fint_{\ep Q}|v|^{p-1}\le \frac{T_0}{R^{sp}}\frac{C}{(1-\ep)^{\frac{(n+sp)^2}{sp}}}\left(\fint_{\ep Q}|v|^{p}\right)^{\frac{p-1}{p}}\\
&\le \frac{C}{(1-\ep)^{\frac{(n+sp)^2}{sp}}}\frac{C}{(1-\ep)^{\frac{(n+sp)n}{s}}}\frac{T_0}{R^{sp}}\left(\sup_{t_0-T_0<t<t_0}\fint_{B_R}|v(x,t)|dx\right)^{p-1}. 
\end{align}
Letting $\sigma = \ep^2$, as well as estimating \[\frac{1}{1-\sqrt{\sigma}}=\frac{1+\sqrt{\sigma}}{1-\sigma}\le \frac{2}{1-\sigma}, \]
we obtain 
\[
\sup_{\sigma Q}v \le \frac{C}{(1-\sigma)^\alpha}\frac{T_0}{R^{sp}}\left(\sup_{t_0-T_0<t<t_0}\fint_{B_R}|v(x,t)|dx\right)^{p-1},
\]
with $\alpha = (n+sp)(n+sp+sn)/sp$. We complete the proof by substituting 
\[v = u+ \left(\frac{R^{sp}}{T_0}\right)^{\frac{1}{p-2}} + \tail_\infty(u_+;x_0,\sigma R,t_0-T_0,t_0).\] 
\end{proof}

Using Theorem \ref{main_thm} and Lemma \ref{trunksubsol} we are able to prove local boundedness of solutions to $\partial_t u+Lu=0$ without any assumption on the sign of $u$. 
\begin{theorem}
Let $u$ be a solution to $\partial_tu+Lu=0$ in $2Q$ where $Q = B_r(x_0)\times(t_0-T_0,t_0)$. 
Then 
\begin{align}\label{supu1}
\sup_{\sigma Q}|u| &\le \frac{C}{(1-\sigma)^\alpha}\left(\left(\frac{R^{sp}}{T_0}\right)^{\frac{1}{p-2}} + \frac{T_0}{R^{sp}}\tail^{p-1}_\infty(u;x_0,\sigma r,t_0-T_0,t_0)\right)\\
&+ \frac{C}{(1-\sigma)^\alpha}\frac{T_0}{R^{sp}}\left(\sup_{t_0-T_0<t<t_0}\fint_{B_R}|u(x,t)|dx \right)^{p-1},\notag
\end{align}
for any $\sigma\in (0,1)$. 
\end{theorem}
\begin{proof}
It is obvious that $-u$ is a solution whenever $u$ is. Thus by Lemma \ref{trunksubsol}, both $u_+$ and $u_-=(-u)_+$ are non negative subsolutions that Theorem \ref{main_thm} is applicable to. The result follows since $|u| = u_++u_-$. 
\end{proof}

\subsection*{Estimation of $\tail_\infty(u;x_0,r,t_0-T_0,t_0)$}
We end with a few remarks on the quantity $\tail_\infty(u;x_0,r,t_0-T_0,t_0)$. If $u$ solves 
\[\partial_tu+Lu=0\text{ in }\Omega\times(t_0-T_0,t_0)\] 
and $\overline{B_r(x_0)}\subset\Omega$, then $\tail(u;x_0,r,t_0-T_0,t_0)$ is bounded since 
\[u\in L^p(t_0-T_0,t_0;W^{s,p}(\R^n)).\]  
On the other hand, if $u$ solves \eqref{eqdata}$, \tail_\infty(u;x_0,r,t_0-T_0,t_0)$ is bounded if and only if 
\begin{equation}\label{gest}
\sup_{t_0-T_0<t<t_0}\int_{\R^n}\frac{|g|^{p-1}(x,t)}{1+|x-x_0|^{n+sp}}dx<\infty
\end{equation}
and 
\begin{equation}\label{uest}\sup_{t_0-T_0<t<t_0}\int_\Omega|u(x,t)|^qdx<\infty,\end{equation}
for some $q\ge p-1$. 
While \eqref{gest} is an assumption on the data, we can only verify \eqref{uest} in a few cases. If $B_r(z)$ is a ball such that $B_{2r}(z)\subset\Omega$, we can prove $\sup_{t_0-T_0<t<t_0}\int_{B_r(z)}|u(x,t)|^pdx<\infty$ using the Cacciopollo inequality. In Lemma \ref{cor_cacc}, we choose $\xi=p-1$ and 
\[
d = \left(\frac{r^{sp}}{T_0}\right)^{\frac{1}{p-2}} + \frac{MT_0}{r^{sp}}\tail^{p-1}(u;z,r,t_1, t_2), 
\]
for $t_1<t_0-T_0<t_0<t_2$.
This allows us to estimate the third term on the right hand side in Lemma \ref{cor_cacc} as 
\begin{align*}
&C\int_{t_1}^{t_2}\left(\sup_{x\in\supp\psi}\int_{\R^n\setminus B_r}\frac{u(y,t)_+^{p-1}}{|x-y|^{n+sp}}dy\int_{B_r}v^{p-1}\phi^p(x,t)dx\right)dt\\
&\le \frac{C}{M}\sup_{t_1<t<t_2}d\int_{B_r}v^{p-1}dx\le \frac{C}{M}\sup_{t_1<t<t_2}\int_{B_r}v^{p}\phi^p(x,t)dx, 
\end{align*}
for appropriate choice of $\psi$. The other terms are naturally bounded. For sufficiently large $M$, we can move this term to the left in the Cacciopollo inequality and obtain 
\begin{equation}\label{localtailsup}
\frac{1}{2p}\sup_{t_0-T_0<t<t_0}\int_{B_{r/2}}v^{p}\phi^p(x,t)dx<\infty.
\end{equation}
To estimate $\sup_{t_0-T_0<t<t_0}\int_{B_r(z)}|u(x,t)|^pdx$ when $B_r(z)$ intersects the boundary $\partial\Omega$, we need to assume more. 
Suppose $|g|\le C_0$ in $B_R\times(t_0-T_0,t_0)$, for some ball $B_R\supset\overline{\Omega}$.  Then it is easy to check that the functions $v_1=(u-C_0)_+$ and $v_2=(u+C_0)_-$ are subsolutions in $B_R\times(t_0-T_0,t_0)$. This allows us to use the Cacciopollo inequality in $B_R$ and obtain \eqref{localtailsup} for any $B_r\subset B_R$, for $v_1$ and $v_2$. This proves \eqref{uest}.

%

\bibliography{Boundedness}{}
\bibliographystyle{plain}

\end{document}